\documentclass{amsart}
\usepackage{hyperref}
\usepackage{amsmath}
\usepackage{amsfonts}
\usepackage{amssymb}
\usepackage{amsthm}
\usepackage{multirow}
\usepackage{lscape}
\usepackage{rotating}
\usepackage{lineno}   

\newcommand{\diam}{{\operatorname{diam}}}

\newcommand{\simp}{{\textrm{simp}}}
\newcommand{\sphere}{{\textrm{sph}}}
\newcommand{\poly}{{\textrm{poly}}}

\usepackage[colorinlistoftodos,bordercolor=orange,backgroundcolor=orange!20,linecolor=orange,textsize=tiny]{todonotes}

\newtheorem{introtheorem}{Theorem}

\newtheorem{theorem}{Theorem}[section]
\newtheorem{proposition}[theorem]{Proposition}
\newtheorem{lemma}[theorem]{Lemma}
\newtheorem{corollary}[theorem]{Corollary}
\newtheorem{remark}[theorem]{Remark}
\theoremstyle{definition}
\newtheorem{definition}[theorem]{Definition}
\newtheorem{example}[theorem]{Example}

\newcommand{\R}{\mathbb R}
\newcommand{\calC}{\mathcal C}
\newcommand{\calS}{\mathcal S}

\newcommand{\f}{\hspace{-1pt}f\hspace{-1.5pt}} 

\DeclareMathOperator{\link}{link}
\DeclareMathOperator{\conv}{conv}
\DeclareMathOperator{\vertices}{vertices}
\DeclareMathOperator{\st}{star}

\title
{Topological prismatoids and small simplicial spheres of large diameter}

\author{Francisco Criado \and
Francisco Santos}

\address[F.~Santos]
{
{Departamento de Matem\'aticas, Estad\'istica y Computaci\'on}\\
{Universidad de Cantabria, E-39005 Santander, Spain.}
E\_mail: francisco.santos@unican.es
}

\address[F.~Criado]
{
Institut f\"ur Mathematik\\
Technische Universit\"at Berlin,
Straße des 17. Juni 136,
10623 Berlin, Germany
E\_mail: criado@math.tu-berlin.de
}

\begin{document}

\thanks
{This work is supported by project MTM2017-83750-P of the Spanish Ministry of Science (AEI/FEDER, UE) and grant EVF-2015-230 of the Einstein Foundation Berlin. Work of F.~Criado is supported by the Berlin Mathematical School}

\begin{abstract}
We introduce topological prismatoids, a combinatorial abstraction of the (geometric) prismatoids recently introduced by the second author to construct counter-examples to the Hirsch conjecture. We show that the ``strong $d$-step Theorem'' that allows to construct such large-diameter polytopes from ``non-$d$-step'' prismatoids still works at this combinatorial level. Then, using metaheuristic methods on the flip graph, we construct four combinatorially different non-$d$-step $4$-dimensional topological prismatoids with $14$ vertices. This implies the existence of $8$-dimensional spheres with $18$ vertices whose combinatorial diameter exceeds the Hirsch bound. These examples are  smaller that the previously known examples by Mani and Walkup in 1980 ($24$ vertices, dimension $11$).

Our non-Hirsch spheres are shellable but we do not know whether they are realizable as polytopes.
\end{abstract}


\maketitle

\section{Introduction}
\label{sec:intro}

One of the most important open questions in polytope theory is how big can the graph-diameter of a polytope $P$ be in terms of its dimension $d$ and number $n$ of facets. The gap between the known lower and upper bounds for this function, that we denote $H_\poly(n,d)$, is extremely big: no polynomial upper bound for it is known, and no polytope is known whose diameter exceeds $1.05(n-d)$.

It has been known for more than 50 years~\cite{KlWa67:d-step} that $H_\poly(n,d) \le H_\poly(2n-2d,n-d)$. In practice, this means that, to answer the diameter question, one can restrict to the case $n=2d$.
The famous Hirsch conjecture from 1957 stated that $H_\poly(n,d) \le n-d$.
The conjecture is now disproved~\cite{San12:counterexample} but known counter-examples to it are still rare.
We call such counter-examples \emph{non-Hirsch polytopes}.

The problem can be addressed topologically, by looking at simplicial \emph{spheres}.
In this context we denote by $H_\sphere(n,d)$ the greatest diameter of the \emph{adjacency graph} of all simplicial $(d-1)$-spheres with $n$ vertices.
We call such a sphere \emph{non-Hirsch} if this diameter exceeds $n-d$.
Since $H_\poly(n,d)$ is known to be attained at some \emph{simple} polytope for every $n$ and $d$, we have that $H_\sphere(n,d) \ge H_\poly(n,d)$: for a simple $P$ attaining $H_\poly(n,d)$, the boundary complex of the polar of $P$ is a sphere showing the inequality.
Even though there is no reason to believe that these two functions coincide for every value of $n$ and $d$, one expects their asymptotics to be similar.
(For example, all known upper bounds for diameters of polytopes hold also for spheres, including the Klee-Walkup result that $H_\sphere(n,d) \le H_\sphere(2n-2d,n-d)$. See Proposition~\ref{prop:dstep}).

The Hirsch bound $H_\poly(n, d) \le n-d$ is only known to hold for $n-d \le 6$~\cite{BrSc:diameter-few-facets}, $n\le 12$~\cite{BDHS13:more-bounds} and $d\le 3$, but the smallest known non-Hirsch polytopes have $d=n-d=20$~\cite{MSW15:improved}.
The smallest known non-Hirsch sphere previous to our work was constructed by Mani and Walkup~\cite{MaWa80:counterexample}. It is a shellable $11$-sphere with $24$-vertices ($d=n-d=12$) shown to be non-polytopal by Altshuler~\cite{Alt85:nonpolytopal}.

The main outcome of our work is the construction of non-Hirsch $(d-1)$-spheres with $d=n-d=9$.

\begin{theorem}
\label{thm:main-intro}
There exist non-Hirsch \ $8$-spheres with $18$ vertices. That is,
\[H_\sphere(18,9) > 9.\]
\end{theorem}

This is close to minimal since the inequality $H_\sphere(n,d) \le n-d$ is known to hold for $n-d\le 5$. This was proved for polytopal spheres by Klee and Walkup~\cite{KlWa67:d-step}, and we show in Section~\ref{ssec:hirsch} (Theorem~\ref{thm:dstep5}) how to modify their proof for non-polytopal ones.

One reason to concentrate on small examples is that from them it is very easy to construct bigger ones. In particular, from the example mentioned in Theorem~\ref{thm:main-intro} one easily derives the following:

\begin{corollary}
\label{coro:ops}
$H_\sphere(2d,d) > d$ for every $d \ge 9$.
\end{corollary}

\begin{proof}
The suspension (or ``double-pyramid'') operation shows $H_\sphere(n+2,d+1) \ge H_\sphere(n,d) +1$.
\end{proof}

Adding connected sums to the suspensions used in this proof we obtain a more refined asymptotic bound (see Theorem \ref{thm:asymptotic}). Applying it to the non-Hirsch sphere of Theorem~\ref{thm:main-intro} gives:

\begin{corollary}
\label{coro:asymptotic_intro}
For every $n$ and $d$,
\[
H_\sphere(n,d) >
      \left\lfloor \frac{n-d}{d}\right\rfloor \cdot
    \left(\left\lfloor \frac{10d}{9}\right\rfloor -1\right)
      \simeq 1.11 (n-d) .
  \]
\end{corollary}

\bigskip

Our construction leading to Theorem~\ref{thm:main-intro} uses the same \emph{prismatoid} technique developed by the second author and present in all non-Hirsch polytopes known so far~\cite{San12:counterexample,MSW15:improved}, but we abstract it to a combinatorial/topological context. Recall that a $d$-prismatoid $P$ is a $d$-dimensional polytope whose vertices lie in two parallel facets. Removing from $\partial P$ the relative interiors of these two facets produces a polyhedral complex homeomorphic to the Cartesian product of a $(d-2)$-sphere and a segment, and with all its vertices in the boundary. This complex, that we assume to be simplicial, is what we call a \emph{topological $(d-1)$-prismatoid}. See Section~\ref{sec:tprismatoids} for details. The \emph{width} of a topological prismatoid $\calC$ is defined to equal  $2$ plus the minimum distance, in the adjacency graph, between facets incident to one and the other component of $\partial \calC$. In this setting we prove the following analogue of \cite[Theorem 2.6]{San12:counterexample}:

\begin{introtheorem}[Theorem \ref{thm:dstep_tprism}]
  Let $\calC$ be a topological prismatoid of dimension $(d-1)$, of width $l$ and with $n>2d$ vertices.
  Assume that its two bases are polytopal.
  Then, there exists a topological $(n-d-1)$-prismatoid $\calC'$ with $2n-2d$ vertices and width at least $l+n-2d$.

In particular, if $l>d$ then $\calC'$ is a non-Hirsch simplicial sphere of dimension $D-1:= n-d-1$, with $N:=2D=2n-2d$ vertices, and of adjacency diameter larger than $N-D$.
\end{introtheorem}

Since this theorem is related to the $d$-step property of Klee and Walkup~\cite{KlWa67:d-step}, we say that topological $(d-1)$-prismatoids of width larger than $d$ are \emph{non-$d$-step}. Our goal is to find non-$d$-step topological prismatoids with $n-d$ as small as possible.

We do this by a \emph{simulated annealing} approach on the graph of \emph{flips} among non-$d$-step prismatoids of a given dimension.
That is, we start with a topological $(d-1)$-prismatoid of width $l >d$ and do flips in it at random, but preserving the width and giving higher probability to the flips that go in the direction of decreasing $n$.
See Section~\ref{ssec:tprismatoids} for the definition and properties of flips in topological prismatoids, and Section~\ref{sec:implementation} for details of our heuristics and implementation.

We choose as starting point the $28$-vertex  prismatoid constructed in \cite[Corollary 2.9]{MSW15:improved}, which has dimension $d=5$ as a polytope, that is, $d-1=4$ as a topological prismatoid.
From it we find thousands of non-$d$-step topological $4$-prismatoids with number of vertices below $28$. In particular, we find four combinatorially different ones with $14$ vertices.
Any of these four implies Theorem~\ref{thm:main-intro}, via Theorem~\ref{thm:dstep_tprism}.
Observe that although the prismatoids are found computationally, the proof that they are non-$d$-step is elementary and can be done by hand. 

The constructed prismatoids are analyzed a bit in Section~\ref{sec:results}.
For the four smallest ones we have checked that they are \emph{shellable}, with respect to a natural notion of shelling of topological prismatoids that we introduce in Section~\ref{ssec:shelling}, and which implies shellability for the resulting non-Hirsch spheres 
mentioned in Theorem~\ref{thm:main-intro}
(see Proposition~\ref{prop:shellable}).
Shellability is necessary for polytopality, but we do not know whether our prismatoids (or spheres) are polytopal.

\subsection*{Acknowledgements}
We thank Mortitz Firsching and Michael Joswig for useful discussions and comments.


\section{Preliminaries}
\label{sec:prelim}

\subsection{Pure simplicial complexes, simplicial spheres}
\label{ssec:complexes}

Here we compile several notions from combinatorial topology.

A \emph{simplicial complex} $\calC$ is a collection of subsets of a finite ground set $V$ (typically, $V=[n]:=\{1,\dots,n\}$), that is closed under taking subsets. It is \emph{pure of dimension $d-1$} (in which case we call it a \emph{$(d-1)$-complex}) if all maximal elements of $\calC$ have the same cardinality, equal to $d$.
  The elements of $\calC$ are called \emph{faces} and maximal faces are \emph{facets}; more specifically, a face of size $i$ is called an \emph{$(i-1)$-face}.
  Some faces have special names according to their cardinality:
  \begin{itemize}
  \item Faces of size $1$ and $2$ are called \emph{vertices} and \emph{edges}; together they form a graph, the \emph{$1$-skeleton} of $\calC$.
  \item Faces of size $d$ and $d-1$ are called \emph{facets} and \emph{ridges}; they also define a graph, called the \emph{adjacency graph} (or \emph{dual graph}) of $\calC$: its vertices are the facets and two facets are \emph{adjacent} if they share a ridge.
 \item Every complex has a face of size $0$, the \emph{empty face}.
  \end{itemize}

We call \emph{Hasse diagram of $\calC$} the Hasse diagram of the partial order of faces by inclusion. That is, it is a directed graph with an arc $f_1 \to f_2$ for every pair of faces $f_1$ and $f_2$ with $f_2=f_1 \cup \{v\}$ for some $v\in V$. Observe that the $1$-skeleton and the adjacency graph of a pure complex contain the information about the lower two and the higher two levels of the Hasse diagram, respectively.

\begin{example}
The boundary complex of a simplicial $d$-polytope is a pure simplicial complex of dimension $d-1$.
\end{example}

A \emph{subcomplex} of $\calC$ is a subset of faces that is itself a complex. The subcomplex \emph{induced} by a subset $W$ of vertices is the set of faces of $\calC$ contained in $W$.
Associated to every face $f\in\calC$ there are the following three (perhaps non-induced) subcomplexes of $\calC$.
 \begin{itemize}
    \item The \emph{deletion} of $f$ in $\calC$ is the set of faces disjoint from $f$. That is, it equals the subcomplex induced by $V\setminus f$.
    \item The \emph{star} of $f$ in $\calC$ is the set of faces $f'$ such that $f\cup f'$ is also a face. Equivalently, it is the simplicial complex whose facets are the facets of $\calC$ containing $f$. Observe that with our definition the star is a \emph{closed} neighborhood of $f$.
    \item The \emph{link} of $f$ in $\calC$ is the set of faces in the star that do not contain $f$. That is, it equals the deletion of $f$ in the star of $f$.
    \end{itemize}
We call \emph{neighborhood} of $f$ the set of vertices of its star.

Every simplicial complex $\calC$ has a realization as a topological space obtained as follows: consider a disjoint family of simplices in $\R^N$ (for sufficiently big $N$) consisting of a simplex of dimension $i$ for each $i$-face in $\calC$. Then take the topological quotient of this set of simplices, by identifying faces as indicated by containment in $\calC$. A \emph{simplicial $(d-1)$-manifold} or \emph{triangulated manifold} (with or without boundary) is a pure $(d-1)$-complex whose realization is a manifold. A \emph{simplicial $(d-1)$-sphere} or \emph{$(d-1)$-ball} is defined in the same way.  A simplicial $(d-1)$-sphere is \emph{polytopal} if it can be realized as the boundary complex of a $d$-polytope.

Every ridge of a simplicial manifold is contained in either two or one facets.
Ridges of the first type are called \emph{interior} and those of the second type are called \emph{boundary}. The boundary of a $(d-1)$-manifold $\calC$ is the $(d-2)$-pure complex having as facets the boundary ridges. That is, all faces contained in boundary ridges are \emph{boundary faces}. The \emph{adjacency graph} of a simplicial manifold $\calC$ is the graph having as nodes the facets of $\calC$ and as edges the pairs of facets that share an interior ridge. 

\subsection{The Hirsch conjecture and its relatives}
\label{ssec:hirsch}

The \emph{(combinatorial) diameter} of a polytope is the diameter (in the sense of graph theory) of its $1$-skeleton.
Let $H_\poly(n,d)$ denote the maximum diameter among all $d$-polytopes with $n$ facets.
The Hirsch conjecture  stated that $H_\poly(n,d)\le n-d$. The first counter-examples have been obtained in~\cite{San12:counterexample, MSW15:improved} and exceed the conjecture by a 5$\%$ or less. (As a historical note: the original conjecture by Hirsch referred to perhaps-unbounded polyhedra, but the unbounded case was disproved in 1967 by Klee and Walkup~\cite{KlWa67:d-step}. Since then, the expression Hirsch conjecture has been used for the bounded case) 

It has been known for long that $H_\poly(n,d)$ is attained at a simple polytope for every $n$ and $d$. In particular, $H_\poly(n,d)$ equals the maximum diameter of the adjacency graphs of simplicial $d$-polytopes with $n$ vertices. Generalizing this a little bit we define $H_\sphere(n,d)$ to be the maximum diameter of the adjacency graphs of simplicial $(d-1)$-spheres with $n$ vertices, and $H_\simp(n,d)$ to be the maximum diameter of the adjacency graphs of pure simplicial $(d-1)$-complexes with $n$ vertices.

The latter is known to be exponential, but the first two are conjectured to be polynomial and perhaps not far from the Hirsch bound.
More precisely, Table~\ref{tab:bounds} sums up the known bounds.
The lower bounds in the last two rows of the table are meant asymptotically: they hold for fixed but sufficiently large $d$ and as $n$ goes to infinity. They follow from the known smallest known non-Hirsch polytope ($n=40$, $d=20$~\cite{MSW15:improved}) and sphere ($n=24$, $d=12$~\cite{MaWa80:counterexample}) via Theorem~\ref{thm:asymptotic} below, whose proof follows from the following considerations.

\begin{table}[ht]
\begin{center}
\begin{tabular}{c|c|c}
    & Lower bound & Upper bound \\ \hline\hline

    $H_\simp(n,d)$ &

    \begin{tabular}{c}
      $\Omega\left(\tfrac{n^{d-1}}{d^2d!}\right)$ \\ 
      (Criado-Newman 2019~\cite{CrNe19:randomized})
    \end{tabular} &
    \begin{tabular}{c}
    $O\left(\tfrac{n^{d-1}}{d!}\right)$ \\
    (Santos 2013~\cite{San13:diameter})
    \end{tabular}   \\  \hline

    $H_{sph}(n,d)$ &
    \begin{tabular}{c}
      $\simeq 1.08(n-d)$ \\
      (Mani-Walkup 1980~\cite{MaWa80:counterexample}) \\
      \hline
      $\simeq 1.11(n-d)$ (This paper)
    \end{tabular} &
    \multirow{2}{*}{
      \begin{tabular}{c}
      $\min\{2^{d-3}n, n^{\log{d-2}} \} $  \\
        (Larman 1970~\cite{Larman70:paths},  \\
        Kalai-Kleitman 1992~\cite{KaKl92:quasip})
      \end{tabular}
    }
  \\ \cline{1-2}

    $H_{poly}(n,d)$ & \begin{tabular}{c}
      $\simeq 1.05(n-d)$ (Matschke-\\
      Weibel-Santos 2017~\cite{MSW15:improved})
    \end{tabular} & \\
\end{tabular}
\end{center}
\caption{Known bounds for maximum diameters of classes of simplicial complexes}
  \label{tab:bounds}
\end{table}

Additionally, the Hirsch bound is known for $2$-spheres (all of which are polytopal) and for spheres with $n-d\le 5$. The latter is stated in~\cite{KlWa67:d-step}  only for polytopes, but the proof works for spheres with minor changes as we sketch below.

In the proof we need the \emph{one-point suspension} construction: the one-point suspension of a simplicial complex $\calC$ at a vertex $w$ is the complex $\calC'$ obtained
by considering the usual suspension $\calC *\{\{w_1\}, \{w_2\}\}$ of $\calC$ and then merging the stars of edges $ww_1$ and $ww_2$ into the star of a single edge $w_1w_2$ so that the vertex $w$ disappears.
For a more direct and explicit description:
\begin{align*}
\calC' :=\ &
\{F, F\cup\{w_1\}, F\cup\{w_2\} , F\cup\{w_1,w_2\} : F \in \link_{\calC}(w)\}
\\
\cup\ &
\{F, F\cup\{w_1\}, F\cup\{w_2\} : w\not\in F \in \calC\}.
\end{align*}
The one-point suspension of a sphere is a sphere with one more dimension, one more vertex and (at least) the same diameter as the original one.  One-point suspensions of polytopal spheres are polytopal. (See details for example in~\cite{KiSa09:companion, San12:counterexample}).

\begin{proposition}[Sphere version of~\protect{\cite[Proposition 2.10]{KlWa67:d-step}}]
\label{prop:dstep}
\[
H_\sphere(n,d) \le H_\sphere(2n-2d,n-d) \qquad \forall n,d.
\]
\end{proposition}

\begin{proof}
Let $\calS$ be a $(d-1)$-sphere with $n$ vertices and let $F$ and $G$ be two facets in it.
If $n=2d$ then there is nothing to prove.

If $n<2d$ we use induction on $2d-n$. Observe that $F$ and $G$ cannot be disjoint, so let $v\in F\cap G$. Denote $\calS'=\link_{\calS}(v)$, and consider $F'=F\setminus \{v\}$ and $G'=G\setminus \{v\}$, which are facets in $\calS'$. The distance from $F$ to $G$ in $\calS$ is clearly bounded by the distance from $F'$ to $G'$ in $\calS'$, and $\calS'$ is a $(d-2)$-sphere with at most $n-1$ vertices. By inductive hypothesis the diameter of $\calS'$ is bounded by $H_\sphere(2(n-1)-2(d-1),(n-1)-(d-1))= H_\sphere(2n-2d,n-d)$.

If $n>2d$ we use induction on $n-2d$. Denote $\calS'$ the one-point suspension of $\calS$. Then, 
\[
\diam(\calS) \le \diam(\calS') \le H_\sphere(n+1,d+1)\le H_\sphere(2n-2d,n-d),
\]
by inductive hypothesis.
\end{proof}

\begin{theorem}[Sphere version of~\protect{\cite[Theorem 4.2]{KlWa67:d-step}}]
\label{thm:dstep5}
For any $n,d$ with $n-d\le 5$,
$H_\sphere(n,d) \le n-d.$
\end{theorem}

\begin{proof}
By Proposition~\ref{prop:dstep} we can assume $n=2d$.
So, let $\calS$ be a $(d-1)$-sphere with $2d$ vertices and let $F$ and $G$ be two facets in it achieving the diameter. We want to show that he distance from $F$ to $G$ to be at most $d$.

The proof follows the one in~\cite{KlWa67:d-step} for the polytopal case and consists of the following three claims. Claims 1 and 2 do not need the assumption that $d\le 5$, and correspond to Theorem 2.8 and Proposition 3.4.(b) in~\cite{KlWa67:d-step}.

\emph{Claim 1: there is no loss of generality in assuming $F\cap G=\emptyset$}.
If this is not the case, let $\calS'=\link_\calS(F\cap G)$. Then, $F':=F\setminus G$ and $G':=G\setminus F$ are facets in $\calS'$ at distance at least equal to the distance between $F$ and $G$ in $\calS$. Observe that $\calS'$ is a $(d-1-k)$-sphere, where $k=|F\cap G|$,  and it has between $2d-2k$ and $2d-1$ vertices.
  If  $\calS'$ has $2d-2k$ vertices, the diameter of $\calS'$ is bounded by $H_\sphere(2d-2k,d-k)$ which, by induction on $d$, is at most $d-k$.
  If $\calS'$ has more than $2d-2k$ vertices, let $\calS''$ be the iterated one-point suspension of $\calS'$ at each vertex $w$ not in $F'\cup G'$.
$\calS''$ is a sphere of dimension $d-k-1+l$ and with $2d-2k+2l$ vertices, where $l$ is the number of times we did the one-point suspension. In $\calS'$, $F'$ and $G'$, each together with one of the two vertices of each suspension, give two complementary facets which are at least at the same distance as $F'$ and $G'$ were in $\calS'$. 

\emph{Claim 2: assume $F\cap G=\emptyset$.
There are vertices $v\in F$ and $w\in G$ such that $\{v,w\}$ is an edge in $\calS$ and such that $F$ and $G$ are adjacent respectively to facets $F'$ and $G'$ with $\{u,v\}\in F'\cap G'$.}
Let $F'$ be any facet adjacent to $F$. Let $w$ be the unique vertex in $F'\setminus F$ and $u$  the unique vertex in $F\setminus F'$.
It is impossible for the $d-1$ facets adjacent to $G$ and containing $w$ to all of them use $u$: indeed, if that happened then these $d-1$ facets together with $G$ form a ball with $w$ in its interior, implying that there are no other facets in the star of $w$; this is impossible because $F'$ is in the star of $w$ too.
Hence, there is a facet $G'$ adjacent to $G$, using $w$, and using a vertex $v$ of $F$ other than $u$. This proves the claim for the edge $vw$ thus obtained.

\emph{Claim 3: with $v$, $w$ as above, there is a path in $\st_\calS(\{v,w\})$ of length at most $d-2$ between a facet $F'$ adjacent to $F$ and a facet $G'$ adajcent to $G$.} The proof of this claim is the complicated part, occupying most of Section 4 (pages 69--71) of~\cite{KlWa67:d-step}. The good thing is that we do not to even check that the proof extends to spheres, since it is a statement about the link of the edge $\{u,v\}$ in $\calS$ and that link is a $(d-3)$-sphere, hence polytopal for $d\le 5$. Indeed, what Klee and Walkup show is that the following (which paraphrases \cite[p. 69, ll.~4--8]{KlWa67:d-step}) holds for every simplicial $k$-polytope $Q$ with $k\le 3$ and with between $2k$ and $2k+2$ vertices: ``if the vertices of $Q$ are divided into two disjoint classes $X$  and $Y$, each consisting of at most $k+1$ vertices, then there is a path of length at most $k$ from a facet contained in $X$ and a facet contained in $Y$.''
This proves the claim by letting $Q$ be a polytope isomorphic to $\link_\calS(\{v,w\})$, $k=d-2$, and  $X$ and $Y$ be $F\setminus \{v\}$ and $G\setminus \{w\}$, respectively.
\end{proof}

We call a $d$-polytope or $(d-1)$-sphere with $n$ vertices \emph{non-Hirsch} if it has diameter $l$ greater than $n-d$. Its \emph{excess} is defined to be $\frac{l}{n-d} - 1$.
From any non-Hirsch sphere or ball infinitely many additional ones can be obtained by the following procedures:

\begin{itemize}
\item The join of a $(d_1-1)$-sphere $\calS_1$ with $n_1$ vertices and diameter $l_1$ and a $(d_2-1)$-sphere $\calS_2$ with $n_2$ vertices and diameter $l_2$ is a $(d_1+d_2-1)$-sphere with $n_1+n_2$ vertices and diameter $l_1+l_2$. It is denoted $\calS_1*\calS_2$.
\item The connected sum of two $(d-1)$-spheres $\calS_1$ and $\calS_2$ with $n_1$ and $n_2$ vertices and of diameters $l_1$ and $l_2$ is a
$(d-1)$-sphere with $n_1 + n_2 - d$ vertices and of diameter at least  $l_1 + l_2 -1$. It is denoted $\calS_1 \# \calS_2$. (Strictly speaking, the diameter of a connected sum of two spheres depends on the choice of facets to glue; we here assume that they are glued in the worst possible way).

\item The suspension of $\calS$ has one more dimension and two more vertices than $\calS$, and its diameter is one more than that of $\calS$.
\end{itemize}

These constructions, which all preserve polytopality, lead to the following:

\begin{theorem}[Variation of \protect{\cite[Theorem 6.5]{San12:counterexample}}]
\label{thm:asymptotic}
If for a certain $d_0$ we know that $H_\sphere(2d_0,d_0) = l_0$, then
\[
H_\sphere(n,d) >
      \left\lfloor \frac{n-d}{d}\right\rfloor \cdot
    \left(\left\lfloor \frac{d}{d_0}\right\rfloor (l_0-d_0) + d-1\right) \simeq \frac{l_0}{d_0}(n-d),
      \qquad
      \forall n , d.
\]
In particular, we have
$ H_\sphere(n,d) \gtrsim (n-d)\frac{l_0}{d_0}$
 for  $n \gg d \gg d_0$.

The same holds for $H_\poly$.
\end{theorem}

\begin{proof}
Let $\calS_0$ be the initial sphere, of dimension $d_0-1$, with $2d_0$ vertices, and with diameter $l_0$. Then, for every $k$ the $k$-fold suspension $\calS^{*k}$ of $\calS$ has dimension $kd_0-1$, diameter $kl_0$, and $2kd_0$ vertices. Letting $k=\lfloor d/d_0\rfloor$ and performing
$d- d_0 k$ suspensions on $\calS^{*k}$ we obtain that
\[
H_\sphere(2d,d) \ge \left\lfloor \frac{d}{d_0}\right\rfloor (l_0-d_0) + d,
\qquad \forall d,k.
\]

Let $T^d$ be the $(d-1)$-sphere on $2d$ vertices obtained so far. By a connected sum of $\lfloor n/d \rfloor-1 = \lfloor (n-d)/d \rfloor$ copies of $T^d$ we conclude that
\begin{align*}
H_\sphere(n,d) \ge H_\sphere(d\lfloor n/d\rfloor, d)
    &\ge
    \left\lfloor \frac{n-d}{d}\right\rfloor \cdot
    \left(\left\lfloor \frac{d}{d_0}\right\rfloor (l_0-d_0) + d\right)
    - \left\lfloor \frac{n-d}{d}\right\rfloor +1 \\
    &=
    \left\lfloor \frac{n-d}{d}\right\rfloor \cdot
    \left(\left\lfloor \frac{d}{d_0}\right\rfloor (l_0-d_0) + d-1\right)
  +1.
\end{align*}
\end{proof}

The smallest non-Hirsch spheres constructed in this paper have dimension $8$, $18$ vertices, and excess $1/9$, which gives the bound in Corollary~\ref{coro:asymptotic_intro}.

\subsection{Prismatoids and the strong $d$-step theorem}
\label{ssec:prismatoids}

The $d$-step theorem of Klee and Walkup is the statement that $H_{\poly}(n,d)\le H_{\poly}(2(n-d),n-d)$ for every $n$ and $d$. In particular, it reduces the study of the Hirsch conjecture or the asymptotic behavior of $H_{\poly}(n,d)$ to the case $n=2d$. The proof works with no change for $H_{\sphere}(n,d)$, since it is purely combinatorial.

Santos' construction of non-Hirsch polytopes is based on a version of this result for a particular class of polytopes, the so-called prismatoids.

\begin{definition}
\label{defi:prismatoid}
  A \emph{prismatoid} is a polytope $Q$ with two parallel facets $Q^+$ and $Q^-$, that we call the \emph{bases}, containing all the vertices. We call a prismatoid \emph{simplicial} if all faces except perhaps $Q^+$ and $Q^-$ are simplices. Observe that the faces of a prismatoid of dimension $d$, excluding the two bases, form a simplicial complex of dimension $d-1$ and homeomorphic to the product of $\mathbb{S}^{d-2}$ with a segment. We call this complex the \emph{prismatoid complex} of $Q$.

The \emph{width} of a prismatoid is the distance from one base to the other, measured in the adjacency graph.
\end{definition}

\begin{theorem}[Strong d-step theorem for prismatoids~\cite{San12:counterexample}]
  If $Q$ is a simplicial $d$-prismatoid of width $l$ and with $n>2d$ vertices, there exists a simplicial $n-d$-prismatoid $Q'$ with $2n-2d$ vertices and width at least $l+n-2d$.

  In particular, if $l>d$ then (the simple polytope dual to) $Q'$ violates the Hirsch conjecture.
\end{theorem}

Santos' original counterexample applies this result to a $5$-prismatoid with $48$ vertices and of width six, thus obtaining a non-Hirsch 23-polytope with 46 facets. This was improved in \cite{MSW15:improved} to a $5$-prismatoid of the same width but with only $25$ vertices, which provides a non-Hirsch polytope in dimension 20.

\section{Topological prismatoids and the topological strong $d$-step theorem}
\label{sec:tprismatoids}

\subsection{Prismatoids and flips in them}
\label{ssec:tprismatoids}

We now define the main object we work with:

\begin{definition}
\label{defi:tprismatoid}
  A \emph{($(d-1)$-dimensional) topological prismatoid} $\calC$ is a $(d-1)$-dimensional pure simplicial complex homeomorphic to $\mathbb{S}_{d-2}\times [0,1]$ (that is, it is homeomorphic to a cylinder), and such that every face with all its vertices in the same boundary component is a boundary face. Put differently, the two boundary components, each homeomorphic to $\mathbb{S}_{d-2}$, are induced subcomplexes.
\end{definition}

Bistellar flips were introduced for general manifolds in \cite{Pachner91:flips} and they are a standard tool in combinatorial topology by now, as local modifications that preserve the PL-type.
The main result of Pachner~\cite{Pachner91:flips} is the converse: every two PL-homeomorphic simplicial manifolds can be transformed into one another via a sequence of bistellar flips.
We here adapt the general definition to the case of topological prismatoids:

\begin{definition}\label{defi:bistellar}
  A \emph{ flip} in a topological prismatoid $\calC$ is a triple $(f,l,v)$ of pairwise disjoint
subsets of  $V(\mathcal{C})$ such that $f$ is a face, $l$ is a minimal nonface, $\link_\calC(f)=\partial(l)*v$, $v$ is either the empty face  or a vertex, and one of the following two things happens:
\begin{itemize}
  \item $|f|+|l|=d+1$ and $v=\emptyset$, in which case $l$ is required to intersect both bases of $\calC$. (Observe that in this case $\link_\calC(f)=\partial(l)*v=\partial(l)$).
  \item $|f|+|l|=d$ and $v$ is a vertex, in which case $f$ and $l$ are required to be contained in the base opposite to $v$.
\end{itemize}
In both cases, the result of the flip is the prismatoid
\[
\mathcal{C'} = \calC \setminus \st_\calC(f) \cup (l * \partial(f) * v).
\]
flips with $v=\emptyset$ are called \emph{interior flips} and flips where $v$ is a vertex are called  \emph{boundary flips}. The \emph{support} of the flip is $f\cup l \cup v$.
\end{definition}

Put differently, an $(f,l,v)$ flip removes all faces containing $f$ and inserts as new faces all subsets of $f\cup l\cup v$ that contain $l$ but not $f$. The interior flips do not change the boundary and are exactly the traditional bistellar flips; the main new feature of our definition is that we allow flips that change the boundary, but guaranteeing the two bases to still be induced subcomplexes after the flip. Observe that a boundary flip can remove or add a vertex. This happens when $f$ or $l$, respectively, have size $1$.

\begin{remark}
  In Definition~\ref{defi:bistellar}, $V(\mathcal{C})$ is understood as the \emph{ground set} of the prismatoid, which may contain points that are not used as vertices. In particular, a boundary flip may have $l=\{w\}$ for a $w$ that is not a vertex, and $f$ a facet in a base. The result of the flip is that this facet is stellarly subdivided with the new vertex $w$.

  We need this type of flips because we want flips to be reversible for the simulated annealing framework. These flips are the inverse of vertex-removing flips.
\end{remark}

An important feature used in our implementation is that knowing only the support $u$ of a flip we can recover the sets $f$, $l$ and $v$ and thus perform the flip:

\begin{itemize}
\item If $u$ has a single vertex from one of the bases then the flip is a boundary flip, and that vertex is $v$. Indeed, in an interior flip $l$ has at least one vertex from each component by definition, and $f$ has at least another from each base because the condition $\link(f)=\partial(l)$, with $|f|+|l|=d+1$, implies that $f$ is an interior face.

\item In both cases, the set $f\cup v$ equals the intersection of all facets of $\calC$ contained in $u$. This allows us to recover $f$, and hence $l$, once we know $v$ by the previous point.
\end{itemize}

The support $u$ of a flip must have $d+1$ vertices, since it is the vertex set of a $d$-ball of the form $ l * \partial(f) * v$, that is, the join of an $i$-simplex and  the boundary of a $j$-simplex, with $i+j=d-1$.

Moreover, the following result allows us to detect flips:

\begin{proposition}
  Given a set $u$ of $d+1$ vertices (or $d$ vertices and an unused point in the case of insertion flips) not all in one base, let $f$, $l$ and $v$ be as above.
The following conditions are necessary and sufficient for $u$ to support a flip in $(f,l,v)$:

\begin{enumerate}
  \item $u$ is the neighborhood of a ridge.
  \label{it:flip1}
\item $\text{neigh}(f)$ has $d+1$ vertices (or $d$ vertices in case of insertion flips).
  \label{it:flip2}
  \item $l$ is not a face of $\calC$.
  \label{it:flip4}
\end{enumerate}
\end{proposition}

\subsection{The strong $d$-step theorem for topological prismatoids}
\label{ssec:dstep}

We here prove the main theoretical result that allows us to use topological prismatoids to search for non-Hirsch spheres. The \emph{width} of $\calC$ is two plus the distance, in the adjacency graph, between the set of facets incident to one base and the set of facets incident to the other (the distance between two sets is, as customary, the minimum distance between respective elements).

\begin{theorem}[Strong d-step theorem for topological prismatoids]
\label{thm:dstep_tprism}
  Let $\calC$ be a topological prismatoid of dimension $(d-1)$, width $l$ and with $n>2d$ vertices.
  Assume that its two bases are polytopal.
  Then, there exists a topological $(n-d-1)$-prismatoid $\calC'$ with $2n-2d$ vertices and width at least $l+n-2d$.

In particular, if \,$l>d$ then $\calC'$ is a simplicial sphere of dimension $D-1:= n-d-1$, with $N:=2D=2n-2d$ vertices whose adjacency graph has diameter larger than $N-D$.
\end{theorem}

\begin{remark}
\rm
The excess of the non-Hirsch sphere produced via Theorem~\ref{thm:dstep_tprism} from a topological $(d-1)$-prismatoid $\calC$ of width $l$ and $n$ vertices equals
\[
\frac{l-d}{n-d}.
\]
Thus, we call that quotient the \emph{(prismatoid) excess} of $\calC$.
\end{remark}

\begin{proof}
  By induction on $n-2d$ it suffices to construct a prismatoid of dimension $d$ with $n+1$ vertices, width at least $l+1$, and such that its bases are polytopal. Repeating this procedure $n-2d$ times we arrive at a $(D-1)$-prismatoid with $2D$ vertices. In such a prismatoid the bases are simplices, so the prismatoid is a $(D-1)$-sphere with $2D$-vertices and diameter at least $l + (n-2d)$.

  For the inductive step, let $B^+$ and $B^-$ be the two bases of $\calC$. Since $\calC$ has more than $2d$ vertices, at least one of them (say $B^+$) is not a simplex. Let $S^+$ be a simplicial polytopal $(d-1)$-sphere containing $B^+$ and with no additional vertices. $S^+$ exists since $B^+$ is polytopal: Let $P\subset\R^d$ be a $d$-polytope realizing $B^+$ and choose a sufficiently generic lifting function $h:\vertices(P) \to \R$. Then $S^+$ can be chosen to be the boundary complex of $\conv\{(v,h(v)) : v\in \vertices(P)\}$. 
  
  Let $S^+_1$ and $S^+_2$ be the two closed $(d-1)$-balls whose intersection is $B^+$ and whose union is $S^+$. Let $v_1$ and $v_2$ be two additional vertices. Consider the following simplicial complex:
\[
\calC':= (\calC \cup B^+_1) * v_1 \cup (\calC \cup B^+_2) * v_2.
\]
$\calC'$ is a topological $d$-prismatoid with bases $S^+$ and the suspension $S^-:=B^- * \{v_1,v_2\}$ of $B^-$ (see Figure~\ref{fig:suspension}). It is not yet the prismatoid we want since it has $n+2$ vertices and we want only $n+1$. Later in the proof we show how to reduce the number of vertices by one, but for the time being let us not care about that. Instead, for reasons that will become apparent later, when computing the length of paths in the adjacency graph of $\calC'$ we will neglect the steps of the form $F*v_1$ to $F*v_2$ or vice versa, for facets $F$ of $\calC$. We claim that, even with this reduced way of counting steps, $\calC'$ has width strictly larger than $\calC$.

\begin{figure}
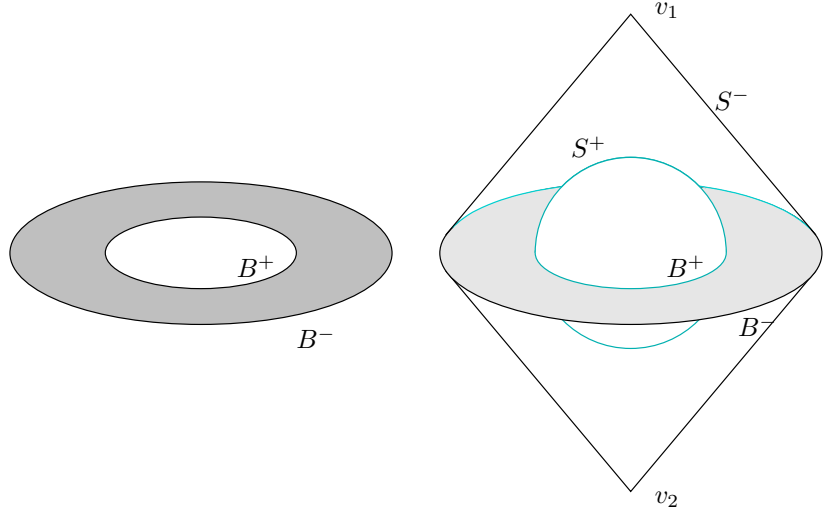
\caption{Sketch of the construction of the prismatoid $\calC'$ (right) from $\calC$ (left) in the proof of Theorem~\ref{thm:dstep_tprism}}
\label{fig:suspension}
\end{figure}

To prove the claim,  let $F'_0,\dots,F'_t$ be a path in $\calC'$ from a facet $F'_0$ adjacent to $S^-$ to a facet $F'_t$ adjacent to $S^+$.
  Since we want our path as short as possible, there is no loss of generality in assuming that $F'_t$ is the only facet adjacent to $S^+$. That is, each facet $F'_i$, $i\in\{0,\dots,{t-1}\}$, is of the form $F_i * v_j$ for a certain facet $F_i$ of $\calC$ and $j\in \{1,2\}$. We have that:
\begin{itemize}
\item Facets $F_i$ and $F_{i+1}$ are either adjacent or the same; the latter happens if and only if $F'_i$ and $F'_{i+1}$ are of the form $F*v_1$ and $f*v_2$ for the same facet $F$ of $\calC$.
\item The first facet $F_0$ is adjacent to $B^-$, since $F'_0$ is adjacent to $S^-= B^- * \{v_1,v_2\}$.
\item The last facet $F_{t-1}$ is adjacent to $B^+$, since $F'_{t-1}$ is obtained from the facet $F'_t =F_t *v_j$ by changing a single vertex, and $F_t$ is a facet of $S^+$, not of $\calC$.
\end{itemize}

Thus, as claimed, the width of $\calC$ is strictly larger than that of $\calC$, even neglecting the steps $F*v_1\leftrightarrow F*v_2$.

We now get rid of one vertex without decreasing the width of $\calC'$. For this,  let $v$ be any vertex of $B^-$ and observe that
\[
\link_{\calC'}(vv_1)=
\link_{\calC'}(vv_2)=
\link_{\calC}(v).
\]
This implies that we can substitute in $\calC'$ the stars of edges $vv_1$ and $vv_2$ by the star of a single edge $v_1v_2$, to obtain a topological prismatoid with one less vertex, that is, with $n+1$ vertices. More precisely, we let:
\[
\calC'':= \calC' \setminus (vv_1 * \link_{\calC}(v) \cup vv_2 * \link_{\calC}(v)) \cup v_1v_2 * \link_{\calC}(v).
\]
restricted to $S^-$, which was the suspension of $B^-$, this operation produces the one point suspension of $B^-$ at vertex $v$. Thus,
$\calC''$ is a prismatoid with bases $S^+$ and the one-point suspension of $B^-$ at $v$. It has $n+1$ vertices, dimension $d$, and it has the same width as $\calC'$ if we neglect the steps of the form $F*v_1\leftrightarrow F*v_2$; in particular, it has width strictly larger than that of $\calC$. $S^+$ is a polytopal $(d-1)$-sphere by the way we constructed it, and the other base is also polytopal since  one-point suspensions of polytopes are polytopes.
\end{proof}

\begin{remark}
\rm
\label{rem:polytopal_bases}
Analyzing the proof of Theorem~\ref{thm:dstep_tprism} the reader can check that the only place where we need the bases of $\calC$ to be polytopal is in order to construct the $(d-1)$-sphere $S^+$ from the base $B^+$ of $\calC$. That is, strictly speaking we do not need each base $B$ to be polytopal but only to be embeddable in a sphere of one more dimension without extra vertices (and we need to keep this property recursively until $B$ becomes a simplex).
We do not know whether all spheres have this property but we suspect not.
\end{remark}

\subsection{Prismatoids of large width via reduced incidence patterns}
\label{ssec:incidence}

In all previous constructions of non-Hirsch polytopes, the proof that the prismatoids to which Theorem~\ref{thm:dstep_tprism} is applied is non-$d$-step uses the following result. 

\begin{proposition}
\label{prop:incidence-pre}
Let $\calC$ be a (geometric or topological) prismatoid with bases $B^+$ and $B^-$. 
A necessary condition for $\calC$ to be $d$-step is that there are vertices $v\in B^+$ and $w\in B^-$ such that $vw$ is an edge and the star of $vw$ contains facets incident to both bases.
\end{proposition}

This proposition is a rephrasing of~\cite[Proposition 2.1]{MSW15:improved}, and used in part (3) of~\cite[Lemma 5.9]{San12:counterexample}.
Observe that the necessary condition in the statement is exactly Claim 2 in the proof of Theorem~\ref{thm:dstep5}. 
Following~\cite{MSW15:improved} we introduce the following graph-theoretical way to visualize this property:

\begin{definition}
\label{defi:incidence}
Let $\calC$ be a topological prismatoid with bases $B^+$ and $B^-$.
The \emph{incidence pattern} of $\calC$  is the bipartite directed graph having a node for each vertex of $\calC$ with bipartition given by the bases and with the following arcs: for each $v\in B^+$ and $w\in B^-$ we have an arc $v\to w$ (resp., $w\to v$) if there is a facet $F$ in $\calC$ containing $vw$ and incident to $B^+$ (resp., incident to $B^-$). 
The \emph{reduced incidence pattern} is the subgraph induced by vertices that are not sources.
\end{definition}

In this language, Proposition~\ref{prop:incidence-pre} becomes:

\begin{proposition}[Topological version of \protect{\cite[Proposition 2.3 ]{MSW15:improved}}]
\label{prop:incidence}
Let $\calC$ be a topological prismatoid. If there is no directed cycle of length two (that is, a ``bidirectional arc'') in its reduced incidence pattern then $\calC$ is non-$d$-step.
\end{proposition}

\begin{proof}
Suppose $\calC$ is $d$-step, so that there is a sequence of facets $F_1,\dots,F_{d-1}$, each adjacent to the next, and with $F_0$ adjacent to $B^+$ and $F_1$ adjacent to $B^-$. In particular, $F_1$ consists of a vertex $v$ of $B^-$ and $d-1$ vertices of $B^+$, and $F_{d-1}$ consists of a vertex $w$ of $B^+$ and $d-1$ vertices of $B^-$. This can only happen if $v$ is already in $F_{d-1}$ and $w$ in $F_1$, so that $v$ and $w$ form a 2-cycle in the reduced incidence pattern.
\end{proof}

%
%

\begin{remark}
The absence of cycles of length two is sufficient but not necessary for being non-$d$-step. For example, two of the four small non-Hirsch prismatoids described in Section~\ref{sec:results} do have cycles of length two in their reduced incidence patterns.
\end{remark}

Using the fact that in a reduced incident pattern without cycles of length two all vertices must have out-degree at least two, the minimum possible patters were classified in \cite{MSW15:improved}. The proof works without changes for topological prismatoids:

\begin{lemma}[\protect{\cite[Proposition 2.4]{MSW15:improved}}]
\label{lemma:patterns}
Let $\calC$ be a topological prismatoid whose reduced
incidence pattern has no cycles of length two. Then, the reduced incident pattern has at least eight vertices.

Moreover, the only two possible patterns with eight vertices are those of Figure~\ref{fig:patterns} (vertices of one base are white, vertices of the other are grey).
\end{lemma}

It is interesting to note that the prismatoids constructed in~\cite{MSW15:improved,San12:counterexample} have the reduced incidence pattern on the left of Figure~\ref{fig:patterns} while the ones we obtain in this paper are related to the pattern on the right. See Section~\ref{sec:results} for details.

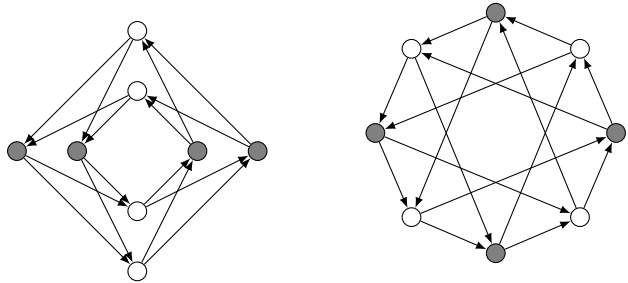
\begin{figure}
\centerline{

\begin{tikzpicture}[mywhite/.style={circle, inner sep=2.5pt, fill=white}, 
mygray/.style={circle, inner sep=2.5pt, fill=gray}, scale =.8 ]

  \node[draw,mygray, label=right:{}] (3) at (2,0) {};
  \node[draw,mygray, label=above:{}] (4) at (1,0) {};
  \node[draw,mygray, label=left:{}] (5) at (-1,0) {};
  \node[draw,mygray, label=below:{}] (6) at (-2,0) {};

  \node[draw,mywhite, label=right:{}] (d) at (0,2) {};
  \node[draw,mywhite, label=right:{}] (e) at (0,1) {};
  \node[draw,mywhite, label=left:{}] (f) at (0,-1) {};
  \node[draw,mywhite, label=left:{}] (g) at (0,-2) {};

  \draw[black, -latex] (3) -- (d);
  \draw[black, -latex] (3) -- (e);
  \draw[black, -latex] (4) -- (d);
  \draw[black, -latex] (4) -- (e);

  \draw[black, -latex] (d) -- (5);
  \draw[black, -latex] (d) -- (6);
  \draw[black, -latex] (e) -- (5);
  \draw[black, -latex] (e) -- (6);

  \draw[black, -latex] (5) -- (f);
  \draw[black, -latex] (5) -- (g);
  \draw[black, -latex] (6) -- (f);
  \draw[black, -latex] (6) -- (g);

  \draw[black, -latex] (f) -- (3);
  \draw[black, -latex] (f) -- (4);
  \draw[black, -latex] (g) -- (3);
  \draw[black, -latex] (g) -- (4);

 \end{tikzpicture}
 \qquad
 \begin{tikzpicture}[mywhite/.style={circle, inner sep=2.5pt, fill=white}, 
mygray/.style={circle, inner sep=2.5pt, fill=gray}, scale =.8 ]

  \node[draw,mygray, label=right:{}] (3) at (2,0) {};
  \node[draw,mygray, label=above:{}] (4) at (0,2) {};
  \node[draw,mygray, label=left:{}] (5) at (-2,0) {};
  \node[draw,mygray, label=below:{}] (6) at (0,-2) {};

  \node[draw,mywhite, label=right:{}] (d) at (1.4,1.4) {};
  \node[draw,mywhite, label=right:{}] (e) at (1.4,-1.4) {};
  \node[draw,mywhite, label=left:{}] (f) at (-1.4,-1.4) {};
  \node[draw,mywhite, label=left:{}] (g) at (-1.4,1.4) {};

  \draw[black, -latex] (3) -- (d);
  \draw[black, -latex] (e) -- (3);
  \draw[black, -latex] (f) -- (3);
  \draw[black, -latex] (3) -- (g);

  \draw[black, -latex] (d) -- (4);
  \draw[black, -latex] (e) -- (4);
  \draw[black, -latex] (4) -- (f);
  \draw[black, -latex] (4) -- (g);

  \draw[black, -latex] (d) -- (5);
  \draw[black, -latex] (5) -- (e);
  \draw[black, -latex] (5) -- (f);
  \draw[black, -latex] (g) -- (5);

  \draw[black, -latex] (6) -- (d);
  \draw[black, -latex] (6) -- (e);
  \draw[black, -latex] (f) -- (6);
  \draw[black, -latex] (g) -- (6);

 \end{tikzpicture}
}
\caption{The two minimal reduced incidence patterns without cycles of length two}
\label{fig:patterns}
\end{figure}

\subsection{Shellability of (topological) prismatoids}
\label{ssec:shelling}

The following concept of shelling for prismatoids is a special case of shelling of  \emph{relative simplicial complexes}, in the sense of~\cite[Section 4.2]{IKN15:balanced}. Indeed, $\calC$ is shellable from $B^+$ to $B^-$ in the sense of Definition~\ref{defi:shellable} if and only if the relative complex $(\calC, B^+)$ is shellable.

\begin{definition}
\label{defi:shellable}
Let $\calC$ be a topological prismatoid with bases $B^+$ and $B^-$. A \emph{prismatoid shelling of $\calC$ from $B^+$ to $B^-$} is an ordering $F_1,\dots, F_K$ of the facets of $\calC$ with the following property: for each $i=1,\dots,K$, the intersection of $|F_i|$ with $|B^+| \cup |F_1|,\cup \dots \cup |F_{i-1}|$ is a $(d-2)$-ball in the boundary complex of $F_i$. Here, the notation $| \cdot |$ applied to a subset of vertices means the subcomplex induced by them.
\end{definition}

\begin{proposition}
\label{prop:shellable}
Shellability is preserved under the ``strong $d$-step construction" of Theorem \ref{thm:dstep_tprism}.
\end{proposition}

\begin{proof}
Let $\calC$ be a prismatoid with polytopal bases and let $\calC'$ be the prismatoid obtained from it in the proof of Theorem \ref{thm:dstep_tprism}. Let $F_1,F_2,\dots,$ is a shelling order of $\calC$.

By the way the sphere $S^+$ is constructed in that proof, there is a line shelling of $S^+$ that completely shells the half-sphere $S^+_1$ first, and then the half-sphere $S^+_2$. The shelling of $\calC'$ is then as follows: 
Following the shelling order of $S^+$, shell first $S^+_1*v_1$, then $S^+_2*v_2$. After that is done, do $F_1*v_1,  F_1*v_2, F_2*v_1, F_2*v_2,\dots,$.
\end{proof}

\begin{remark}
It is not clear to us whether the bases of a prismatoid that is shellable in the sense of Definition~\ref{defi:shellable}  have to be shellable themselves.
\end{remark}

\section{Metaheuristics and implementation}
\label{sec:implementation}

In this section, we show our approach to find non-$d$-step topological prismatoids with few vertices.

The general idea is to start with a $28$ vertices prismatoid as defined in~\cite{MSW15:improved}, and perform  flips on it attempting to remove vertices while preserving its width. A general, well known framework to do this is \emph{simulated annealing}.
We take this one, instead of the smaller one with $25$ vertices also constructed in \cite{MSW15:improved}, as a starting point because it  has much more symmetry.

Simulated annealing is a very common metaheuristic algorithm for optimization problems, used when we have a search space and an ``adjacency relation'' between pairs of feasible solutions.
The idea is to perform a random walk through the state graph of a problem, but favoring moves that improve the desired objective function over moves that do not.
 It has been used successfully in combinatorial topology to simplify simplicial complexes while preserving a condition (typically their homeomorphism type)~\cite{BjLu00:bistellar} or, in conjunction with other strategies, to tackle the problem of sphere recognition~\cite{JLT14:recognition, JLT14:recognition-abstract}.

There is a variable, the \emph{temperature}, regulating the probability assigned to each possible step as a function of how much it improves or worsens the objective. At higher temperatures the  choice is more random;  when the system cools down it converges to accepting only improving moves. Formally, the probability of accepting a step that increases cost by $\Delta c$ at temperature $t$ is:
\[
\left\{
  \begin{array}{cc}
    1 & \text{ if } \Delta c <0, \\ 
     \exp(-\Delta c/t) & \text{ if } \Delta c \geq 0.
  \end{array}\right.
\]
Note that it is very important to choose the potential step with uniform probability, among all the neighbor of the current state. This ``a priori'' probability distribution, together with the cooling schedule, produces an ``a posteriori'' probability distribution of performing the step. This gives higher probability to improving steps, but also gives chance to worsening steps at high temperatures.

That is, areas of the graph with smaller values of the objective function are more likely to be explored. As the temperature cools down, the random walk will focus on these areas and make optimizations with more detail. Loosely speaking, the first few iterations of the algorithm are more exploration-focused, and the last iterations are exploitation-focused.

Formally, simulated annealing requires the following aspects to be decided:
\begin{itemize}
  \item A \emph{state graph} representing the feasible states of the problem and an adjacency relation, plus an initial state.
  \item A \emph{cooling schedule}, that defines temperature as a function of time,   thus modulating the probability of acceptance of a cost-increasing step.
  \item An appropriate \emph{objective function} that we aim to minimize.
\end{itemize}

In our particular problem, our \emph{state graph} consists of all non-$d$-step topological $4$-prismatoids, with an edge between a pair of prismatoids if they differ by a flip.
This graph is undirected, since every flip is reversible by another flip.

\medskip
There is a lot of research on \emph{cooling schedules} for different problems. It is known that SA converges to the global optimum for a certain cooling schedule \cite{BerTsi1993}, but it is too slow for any practical application. Since the best schedule depends on the problem, several adaptative schedules have been proposed too \cite{Ing12:adaptativeSA}. However, the most common approach is to define a geometric cooling schedule, of the form $T_t=t_0*e^{st}$, where the parameters $t_0$ (initial temperature), $s$ (cooling speed) and the number of iterations are adjusted manually. Since the flipping operation is very fast, we have chosen a slow schedule with a high number of iterations.

The particular parameters have been obtained by trial and error. We used cooling schedule $T(k)=1000\cdot 0.99997^k$ and $500000$ iterations for each run.

\subsection{The objective function}
\label{ssec:objective}

The objective function guides our algorithm towards non-$d$-step topological prismatoids with few vertices. This is, prismatoids with many vertices should have higher cost.

A naive approach would be to just take the number of vertices as an objective function. But this objective function has large \emph{plateaus}, connected subgraphs of the state graph with constant number of vertices, and it does not push our state towards less vertices. So we have to find a way to break ties between prismatoids with the same number of vertices by giving less cost to prismatoids from which we expect it to be easier to remove vertices.
That is, the objective function we want has the number of vertices as a main component, plus a smaller (heuristic) \emph{tie-breaker} that pushes the algorithm in the right direction.

A flip that removes a vertex must be of type $(1,d)$; therefore, in order to perform it we must have a vertex with exactly $d+1$ neighbors, which is the smallest possible size for the neighborhood of a vertex in a prismatoid. So, a good approach is to try to reduce the size of the neighborhood of a vertex until it is precisely $d+1$.

However, just taking the size of the smallest neighborhood as a tie-breaker is not sensitive
enough because the algorithm will then not try to reduce the neighborhoods of other vertices.
For this reason, we use as a tie-breaker a generalized mean of the sizes of neighborhoods of vertices.
More precisely, the objective function that achieved the best performance in practice among the ones we tried is
\[
 \text{cost}(\mathcal{C}))= |V(\mathcal{C})| + \varepsilon\left(\frac{\sum_{v\in V(\mathcal{C})} |\text{neigh}(v)|^{-3}}{|V(\mathcal{C})|}\right)^{-1/3}.
 \]

\subsection{Data structures}
\label{ssec:data_str}

A proper ``topological prismatoid'' data structure for our problem needs to allow for the following operations:

\begin{itemize}
  \item Construction from the list of facets.
  \item Check if a set of vertices is a face.
  \item Iterate through the maximal subfaces of a face.
  \item Iterate through the minimal superfaces of a face.
  \item Perform a flip.
  \item Compute the width of the prismatoid.
  \item Get a valid random flip (with uniform probability).
\end{itemize}

We implement a prismatoid $\calC$ as a map of pairs (face,neighborhood), indexed by the faces in $\calC$. Faces and neighborhoods themselves are of type ``set of integers''. We also store the bases of $\calC$ in the same manner.

Observe this implicitly gives us the Hasse diagram (maximal subfaces and minimal superfaces): Each face $F$ is directly above those of the form $F\setminus\{v\}$ for $v\in F$, and directly below those of the form $F\cup \{w\}$ for $w\in\text{neigh}(F)$.
There is no need to store the adjacency graph, because it is implicit in the Hasse diagram.
The facets adjacent to a facet $F$ are computed from the neighborhood of the non-boundary ridges in $F$. Boundary and interior ridges are distinguished by their neighborhoods having $d+1$ and $d$ elements, respectively.

To compute and update the width we store, for each facet, the distance to the first base (chosen arbitrarily but once and for all) and the number of paths achieving that distance. In this way we do not need to explore again all facets to compute the new width after performing a flip, we just need to update the values that change. For this, when we perform the flip the new facets are inserted into a queue, and the distances are updated by cascading through the prismatoid.

\medskip

A flip is implemented simply by removing the old faces and inserting the new inserted faces. What is not so straightforward, and needs to be addressed, is how to implement an unbiased generation of random flips among all the possible ones.

For this, we imitate to some extent the technique used in polymake \cite{GaJo97:polymake}. In polymake, there is a set of pairs $(f,l)$, called ``options'' satisfying some conditions for flipability, in particular conditions \eqref{it:flip1} and \eqref{it:flip2}  of Section~\ref{ssec:tprismatoids}.
The flips are categorized by dimension of $f$. But the list of candidate pairs $(f,l)$ is very hard to update after a flip is performed. Among other things, some potential flips may change their $f$ and $l$ while preserving their support $f\cup l$.

Since the support of every flip is the neighborhood of a ridge, one could simplify this by using the list of ridges instead of the pairs $(f,l)$ as input to generate a random flip. But choosing randomly from the list of ridges creates bias: some flips are more likely than others since several ridges (actually $|l|$ of them) correspond to the same flip.

To avoid these drawbacks we store and update the list of  \emph{ridge-neighborhoods}. That is, for each ridge $F$ we store the facet $F_1$ containing $F$ or the union $F_1\cup F_2$ of the two facets containing $F$ depending on whether $F_1$ is in the boundary or the interior.
This is very easy to update, and it also makes it very easy to spot vertex-adding flips (which correspond to boundary ridge-neighborhoods and are characterized by having $d$ instead of $d+1$ elements). Since there is a bijection between flip-defining ridge-neighborhoods and flips, via the $(f,l,v)$ formalism introduced in Section~\ref{ssec:tprismatoids}, it is easy to choose flips uniformly at random: choose a random ridge-neighborhood and discard non-valid ones.

We find this approach more stable and requiring less changes to the data structure than the ones based on $(f,l)$ pairs or in ridges alone.

\section{Results. Small non-$d$-step prismatoids and spheres}
\label{sec:results}

As said in the previous section,  we ran our algorithm with cooling schedule $T(k)=1000\cdot 0.99997^k$ and $500000$ iterations for each run. We let it run for three days on an openSuse 42.3 Linux machine with 16 GB of RAM and an AMD Phenom X6 1090T processor, after which we had concluded $4093$ runs. We thus obtained $4093$ non-$d$-step topological $4$-prismatoids, with number of vertices ranging between $14$ and $28$. Figure~\ref{fig:histograms} shows the distribution we obtained for the number of vertices alone (top) and for number of vertices versus number of facets (bottom).

\begin{figure}
  \includegraphics[width=0.8\linewidth]{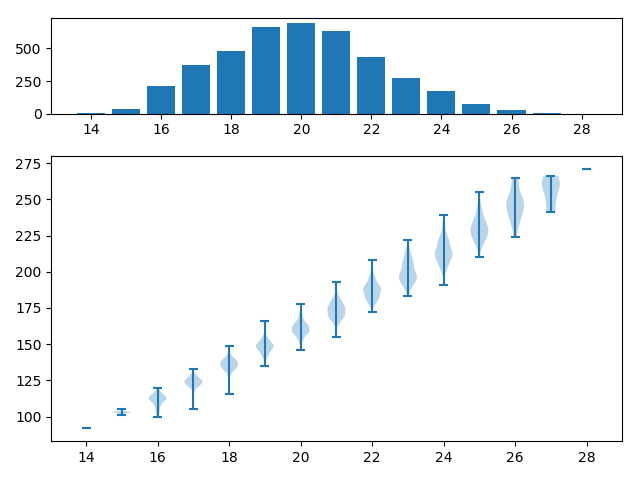}
\caption{Top: distribution of the $4093$ prismatoids by number of vertices. Bottom: distribution by number of vertices and facets. The initial prismatoid has $28$ vertices and $272$ facets.}
\label{fig:histograms}
\end{figure}

For the rest of this section we focus on the $4$ smallest examples, with $14$ vertices. 
We call them $\#1039$, $\#1963$, $\#2669$ and $\#3513$ since these are their indices among the $4093$ experiments that we did. 
They are listed in Tables~\ref{table:1039}--\ref{table:3513}. Vertices from one base are labeled $0$ to $6$
and vertices from the other are $a$ to $g$. In the tables, the facets of each prismatoid are grouped by layers, where a layer consists of all facets sharing the number of vertices they have from each  base.
It is remarkable that the four examples obtained have a lot of similarities:

\newcommand{\prismatoide}[8]{
  \begin{array}{c|ccc|ccc|c}
      \begin{array}{c} #1 \end{array}&
      \begin{array}{c} #2 \end{array}&
      \begin{array}{c} #3 \end{array}&
      \begin{array}{c} #4 \end{array}&
      \begin{array}{c} #5 \end{array}&
      \begin{array}{c} #6 \end{array}&
      \begin{array}{c} #7 \end{array}&
      \begin{array}{c} #8 \end{array}\\
  \end{array}
}

\begin{table}
    \footnotesize
$
\prismatoide
{
    0256g \\
    0245f \\
    1256g \\
    1245f \\
    0234e \\
    0123d \\
    1234e \\
    0126d \\
    0156g \\
    0145f \\
    0134e \\
          \\
}
{
    126cg \\
    015bg \\
    015bf \\
    014ae \\
    013ae \\
    013ad \\
    016cd \\
    016cg \\
    014bf \\
    014ad \\
    014cd \\
    014bc \\
}{
    123ae \\
    124af \\
    123af \\
    123bf \\
    123bg \\
    123cd \\
    123cg \\
    026bg \\
    026cd \\
    026ce \\
    026be \\
    126cd \\
}{
    025af \\
    025bg \\
    024af \\
    024ae \\
    025ae \\
    025be \\
    125bg \\
    125bf \\
    023ce \\
    023cd \\
    124ae \\
          \\
}
{
    14abf \\
    13abf \\
    13acd \\
    13acg \\
    13abg \\
    06bce \\
    04abd \\
    04bcd \\
    14bcg \\
    14abg \\
    14acg \\
    14acd \\
}{
    26abf \\
    26ace \\
    26abe \\
    05abf \\
    03ace \\
    05ace \\
    03acd \\
    05acd \\
    05cde \\
    05bde \\
    05abd \\
    04abf \\
}{
    25abf \\
    25abe \\
    01bcg \\
    06bcg \\
    26bcg \\
    23bcg \\
    23bcf \\
    23acf \\
    23ace \\
    26bcf \\
    26acf \\
          \\
}
{
    5acde \\
    5abde \\
    3bcfg \\
    3acfg \\
    3abfg \\
    0bcde \\
    6bcef \\
    6acef \\
    6abef \\
    4abcg \\
    4abcd \\
          \\
}
$

    \smallskip
    \caption{Prismatoid \#1039}
    \label{table:1039}
$
\prismatoide
{
    0126d \\
    0123d \\
    0134e \\
    0234e \\
    1234e \\
    0145f \\
    0245f \\
    1245f \\
    0156g \\
    0256g \\
    1256g \\
          \\
}
{
    014ae \\
    024ae \\
    013ae \\
    014af \\
    024af \\
    025af \\
    025ae \\
    124af \\
    013af \\
    016bd \\
    026bd \\
    013bd \\
}{
    023bd \\
    123bd \\
    023be \\
    026be \\
    123be \\
    124be \\
    124bd \\
    016bg \\
    013bg \\
    126cd \\
    124cd \\
    124ac \\
}{
    126cg \\
    125cg \\
    125cf \\
    015cf \\
    013cf \\
    015cg \\
    025cg \\
    025ce \\
    026cg \\
    026ce \\
    013cg \\
          \\
}
{
    14abe \\
    24abe \\
    24abd \\
    25abe \\
    25abd \\
    13abe \\
    03abe \\
    06abe \\
    03abf \\
    06abf \\
    13abg \\
    14abg \\
}{
    24acd \\
    12acf \\
    25acf \\
    25acd \\
    05acf \\
    05ace \\
    06ace \\
    06acf \\
    13acf \\
    16bcd \\
    14bcd \\
    26bcd \\
}{
    25bcd \\
    25bce \\
    26bce \\
    13acg \\
    14acg \\
    16bcg \\
    06bcg \\
    06bcf \\
    03bcf \\
    14bcg \\
    03bcg \\
          \\
}
{
    5abde \\
    6abef \\
    3abfg \\
    5acde \\
    6acef \\
    4abcd \\
    5bcde \\
    6bcef \\
    3acfg \\
    4abcg \\
    3bcfg \\
          \\
}
$

    \smallskip
    \caption{Prismatoid \#1963}
    \label{table:1963}
$
\prismatoide
{
    0156g \\
    0256g \\
    1256g \\
    0123d \\
    0126d \\
    0134e \\
    1234e \\
    0234a \\
    0145f \\
    0245f \\
    1245f \\
          \\
}
{
    023ad \\
    013ad \\
    026ad \\
    013ae \\
    034ae \\
    234ae \\
    123ae \\
    124ae \\
    026bg \\
    025bg \\
    125bg \\
    016cg \\
}{
    126cg \\
    015cg \\
    016cd \\
    126cd \\
    123cd \\
    123cg \\
    123ag \\
    124ag \\
    124bg \\
    015cd \\
    015ad \\
    015ae \\
}{
    024af \\
    026af \\
    026bf \\
    025bf \\
    125bf \\
    124bf \\
    015bf \\
    015be \\
    014bf \\
    014be \\
          \\
          \\
}
{
    24abg \\
    23abg \\
    13acg \\
    14acg \\
    13acd \\
    23acd \\
    26acd \\
    06acd \\
    14acd \\
    05acd \\
    05ace \\
    06ace \\
}{
    06bcg \\
    26bcg \\
    23bcg \\
    05bcg \\
    15bcg \\
    14bcg \\
    14bcd \\
    15bcd \\
    05bce \\
    06bce \\
    24abf \\
    23abf \\
}{
    04abf \\
    04abe \\
    06abe \\
    14abe \\
    14abd \\
    15abd \\
    15abe \\
    06abf \\
    23acf \\
    26acf \\
    23bcf \\
    26bcf \\
}
{
    5abde \\
    5acde \\
    5bcde \\
    4abcg \\
    4abcd \\
    3abfg \\
    6abef \\
    3acfg \\
    6acef \\
    3bcfg \\
    6bcef \\
          \\
}
$

    \smallskip
    \caption{Prismatoid \#2669}
    \label{table:2669}
$
\prismatoide
{
    0156g \\
    0256g \\
    1256g \\
    0134e \\
    0234e \\
    1234e \\
    0126d \\
    0123d \\
    0145f \\
    0245f \\
    1245f \\
          \\
}
{
    015ag \\
    025bg \\
    125bg \\
    026bg \\
    126bg \\
    016bg \\
    016bd \\
    026bd \\
    025bd \\
    015af \\
    014af \\
    014ag \\
}{
    014bg \\
    024af \\
    024ae \\
    124af \\
    124ae \\
    123ae \\
    025af \\
    025ae \\
    123af \\
    125bf \\
    126bf \\
    014bc \\
}{
    014ce \\
    013ce \\
    023ce \\
    025ce \\
    025cd \\
    023cd \\
    013cd \\
    123cd \\
    126cd \\
    126cf \\
    123cf \\
          \\
}
{
    04abg \\
    05abg \\
    15abg \\
    05abd \\
    04abd \\
    15abf \\
    25abf \\
    25abe \\
    26abf \\
    26abe \\
    13abf \\
    13abg \\
}{
    14bcg \\
    14acg \\
    14ace \\
    04ace \\
    05ace \\
    13ace \\
    13acg \\
    13bcg \\
    23ace \\
    05acd \\
    04acd \\
    01bcd \\
}{
    04bcd \\
    16bcd \\
    26bcd \\
    25bcd \\
    25bce \\
    26bce \\
    26ace \\
    26acf \\
    23acf \\
    16bcf \\
    13bcf \\
          \\
}
{
    5abde \\
    3abfg \\
    6abef \\
    4abcg \\
    4abcd \\
    5acde \\
    5bcde \\
    3acfg \\
    3bcfg \\
    6acef \\
    6bcef \\
          \\
}
$
    \smallskip
    \caption{Prismatoid \#3513}
    \label{table:3513}
\end{table}

\begin{itemize}
\item They have combinatorially isomorphic bases. Indeed, in all cases the list of facets of the bases are as follows. (To relate this to the tables, observe that the bases correspond to the first and last layer in the prismatoid, removing from each facet the unique vertex from the other base).
\[
\footnotesize
\begin{array}{ccccc}
\small
 &   012 3
    \\ 
        01 34 &
        02 34 &
        12 34
    \\
        01 45 &
        02 45 &
        12 45
    \\
        01 56 &
        02 56 &
        12 56
    \\
&    012 6
\end{array}
\quad\quad\quad
\text{\normalsize and }
\quad\quad\quad
\begin{array}{lll}
&    abc d \\
        ab de &
        ac de &
        bc de
    \\
        ab e\f &
        ac e\f &
        bc e\f
    \\
        ab \f g &
        ac \f g &
        bc \f g
    \\
&    abc g.
\end{array}
\]
Observe that both are the face complex of the stacked $4$-polytope with seven vertices.
That is, they are the boundaries of the stacked $4$-balls $\{01234, 01245, 01256\}$ and $\{abcde, abce\f, abc\f g\}$, respectively.

\item They have the same $f$-vector $(14,85,220,241,92)$. 

\item They are all shellable, with a shelling that is monotone on layers.
That is, no facet of one layer is used until finishing the previous layer.
In the tables, facets within each layer are given in a shelling order.

\item The vector of number of facets in different layers is the same $(11, 35, 35, 11)$, and symmetric, in three of them. In \#2669 we get the slightly asymmetric vector $(11, 34, 36, 11)$.

 \item
Their reduced incidence patterns, shown in Figure~\ref{fig:ifrightful_incidence}, are very similar.
In \#1963 and \#3513 the reduced incidence patterns coincide with the one in the right part of Figure \ref{fig:patterns}, minimal by Lemma~\ref{lemma:patterns}.
In \#1039 and \#2669 they are are almost the same, except each of them has a single ``outlier'' facet incident to a base ($0bcde$ in \#1039 and $0234a$ in \#2669) that introduces a new vertex in the pattern and creates directed cycles of length two. 
In particular, in these two Proposition~\ref{prop:incidence} is not enough to prove the non-$d$-step property.

Note that the starting prismatoid of the algorithm had \emph{the other} reduced incidence pattern of minimal size, the one in the left in Figure \ref{fig:patterns}.
\end{itemize}

We do not know whether any of the four is polytopal.

\begin{figure}
\begin{center}
\quad
\begin{tikzpicture}[mywhite/.style={circle, inner sep=2.5pt, fill=white}, 
mygray/.style={circle, inner sep=2.5pt, fill=gray}, scale =.7 ]

  \node[draw,mywhite, label=above:{3}] (3) at (1.4,1.4) {};
  \node[draw,mywhite, label=above:{4}] (4) at (-1.4,1.4) {};
  \node[draw,mywhite, label=below:{5}] (5) at (-1.4,-1.4) {};
  \node[draw,mywhite, label=below:{6}] (6) at (1.4,-1.4) {};

  \node[draw,mygray, label=above:{d}] (d) at (0,2) {};
  \node[draw,mygray, label=left:{e}] (e) at (-2,0) {};
  \node[draw,mygray, label=below:{f}] (f) at (0,-2) {};
  \node[draw,mygray, label=right:{g}] (g) at (2,0) {};  

%

  \draw[black, -latex] (3) -- (d);
  \draw[black, -latex] (3) -- (e);
  \draw[black, -latex] (f) -- (3);
  \draw[black, -latex] (g) -- (3);

  \draw[black, -latex] (d) -- (4);
  \draw[black, -latex] (4) -- (e);
  \draw[black, -latex] (4) -- (f);
  \draw[black, -latex] (g) -- (4);

  \draw[black, -latex] (d) -- (5);
  \draw[black, -latex] (e) -- (5);
  \draw[black, -latex] (5) -- (f);
  \draw[black, -latex] (5) -- (g);

  \draw[black, -latex] (6) -- (d);
  \draw[black, -latex] (e) -- (6);
  \draw[black, -latex] (f) -- (6);
  \draw[black, -latex] (6) -- (g);

  \node[label=right:{$\#1963\quad \#3513$}] (x) at (-2.1,-3.5) {};
 \end{tikzpicture}
 \ 
\begin{tikzpicture}[mywhite/.style={circle, inner sep=2.5pt, fill=white}, 
mygray/.style={circle, inner sep=2.5pt, fill=gray}, scale =.7 ]

  \node[draw,mygray, label=above:{d}] (3) at (0,2) {};
  \node[draw,mygray, label=left:{e}] (4) at (-2,0) {};
  \node[draw,mygray, label=below:{f}] (5) at (-0,-2) {};
  \node[draw,mygray, label=right:{g}] (6) at (2,0) {};

  \node[draw,mywhite, label= 45:{$0$}] (b) at (0,0) {};
  \node[draw,mywhite, label=above:{4}] (f) at (-1.4,1.4) {};
  \node[draw,mywhite, label=below:{5}] (e) at (-1.4,-1.4) {};
  \node[draw,mywhite, label=below:{6}] (d) at (1.4,-1.4) {};
  \node[draw,mywhite, label=above:{3}] (g) at (1.4,1.4) {};  

%

  \draw[black, -latex] (3) -- (b);
  \draw[black, -latex] (b) -- (3);
  \draw[black, -latex] (d) -- (3);
  \draw[black, -latex] (3) -- (e);
  \draw[black, -latex] (g) -- (3);
  \draw[black, -latex] (3) -- (f);

  \draw[black, -latex] (4) -- (b);
  \draw[black, -latex] (b) -- (4);
  \draw[black, -latex] (4) -- (d);
  \draw[black, -latex] (4) -- (e);
  \draw[black, -latex] (f) -- (4);
  \draw[black, -latex] (g) -- (4);

  \draw[black, -latex] (b) -- (5);
  \draw[black, -latex] (5) -- (d);
  \draw[black, -latex] (e) -- (5);
  \draw[black, -latex] (f) -- (5);
  \draw[black, -latex] (5) -- (g);

  \draw[black, -latex] (b) -- (6);
  \draw[black, -latex] (d) -- (6);
  \draw[black, -latex] (e) -- (6);
  \draw[black, -latex] (6) -- (f);
  \draw[black, -latex] (6) -- (g);

  \node[label=right:{$\#1093$}] (x) at ((-1.2,-3.5) {};
 \end{tikzpicture}
 \ 
\begin{tikzpicture}[mywhite/.style={circle, inner sep=2.5pt, fill=white}, 
mygray/.style={circle, inner sep=2.5pt, fill=gray}, scale =.7 ]

  \node[draw,mywhite, label=above:{3}] (3) at (1.4,1.4) {};
  \node[draw,mywhite, label=above:{4}] (4) at (-1.4,1.4) {};
  \node[draw,mywhite, label=below:{5}] (5) at (-1.4,-1.4) {};
  \node[draw,mywhite, label=below:{6}] (6) at (1.4,-1.4) {};

  \node[draw,mygray, label=right:{$a$}] (a) at (0,0) {};
  \node[draw,mygray, label=above:{d}] (g) at (0,2) {};
  \node[draw,mygray, label=left:{e}] (f) at (-2,0) {};
  \node[draw,mygray, label=below:{f}] (e) at (0,-2) {};
  \node[draw,mygray, label=right:{g}] (d) at (2,0) {};

%

  \draw[black, -latex] (3) -- (a);
  \draw[black, -latex] (a) -- (3);
  \draw[black, -latex] (d) -- (3);
  \draw[black, -latex] (e) -- (3);
  \draw[black, -latex] (3) -- (f);
  \draw[black, -latex] (3) -- (g);

  \draw[black, -latex] (4) -- (a);
  \draw[black, -latex] (a) -- (4);
  \draw[black, -latex] (d) -- (4);
  \draw[black, -latex] (4) -- (e);
  \draw[black, -latex] (4) -- (f);
  \draw[black, -latex] (g) -- (4);

  \draw[black, -latex] (a) -- (5);
  \draw[black, -latex] (5) -- (d);
  \draw[black, -latex] (5) -- (e);
  \draw[black, -latex] (f) -- (5);
  \draw[black, -latex] (g) -- (5);

  \draw[black, -latex] (a) -- (6);
  \draw[black, -latex] (6) -- (d);
  \draw[black, -latex] (e) -- (6);
  \draw[black, -latex] (f) -- (6);
  \draw[black, -latex] (6) -- (g);

  \node[label=right:{$\#2669$}] (x) at (-1.2,-3.5) {};
 \end{tikzpicture}
\end{center}
\caption{The reduced incidence patterns of the four smallest non-$d$-step $4$-prismatoids}
\label{fig:ifrightful_incidence}
\end{figure}
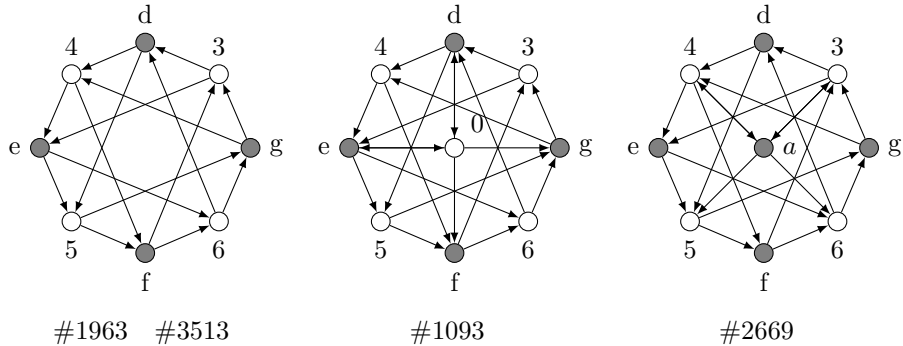

\end{document}